\documentclass{amsart}
\usepackage{amssymb,amsmath,amsthm,graphicx,amscd,multind,eufrak,hyperref,}
\theoremstyle{plain}
\newtheorem{thm}{Theorem}

\newtheorem{prop}[thm]{Proposition}
\newtheorem{cor}[thm]{Corollary}

\newtheorem{problem}[thm]{Problem}
\theoremstyle{definition}

\newtheorem{example}[thm]{Example}

\theoremstyle{definition}

\begin{document}
\title[The Data Singular and the Data Isotropic Loci for Affine Cones]{\bf
The Data Singular and the Data Isotropic Loci for Affine Cones}

\author{Emil Horobe\c{t}}
\address{\textit{Address}: Eindhoven University of Technology, Department of Mathematics  and Computer Science, Eindhoven, The Netherlands.}
\email{e.horobet@tue.nl}


\maketitle
\begin{abstract}
The generic number of critical points of the Euclidean distance function from a data point to a variety is called the Euclidean distance degree. The two special loci of the data points where the number of critical points is smaller then the ED degree are called the Euclidean Distance \textit{Data Singular Locus} and the Euclidean Distance \textit{Data Isotropic Locus}. In this article we present connections between these two special loci of an affine cone and its dual cone.
\end{abstract}

\section{Introduction}

Models in science are often expressed as real solution sets of systems of polynomial equations, namely real algebraic varieties. One of the most fundamental optimization problems that can be formulated on such sets is the following: given a real algebraic variety and given a general data point of the ambient space, minimize the Euclidean distance from the given data point to the variety.

In order to solve this problem algebraically we examine the critical points of the squared Euclidean distance function. The number of such critical points expresses the algebraic degree of the complexity of writing the optimal
solution to the distance minimization problem and it is called the \textit{Euclidean Distance Degree.}
This optimization problem arises in a wide range of applications, such as low rank approximations (Example~\ref{Ex:Determinant}), control theory (Example~\ref{Ex:Hurwitz}), formation control (Example~\ref{Ex:Cayley-Menger}), algebraic statistics (Example~\ref{Ex:Cayley cubic}) and multiview geometry (Example~\ref{Ex:Essential}).

For a general data point $u$ the number of complex critical points is constant while, the number of real critical points is typically not constant for all general $u$. For example, if one of the critical points has a multiplicity, then the number of real critical points typically changes, this locus is called the \textit{ED-discriminant} (or classically \textit{focal loci}) and it was studied in \cite{CT00,DHOST13,GKZ,JP14,T98}.

In this article we want to discuss the locus (different from the ED discriminant) of exceptional data points $u$ for which the number of complex critical points is smaller then the ED degree. There are three ways in which we can have different number of critical points than expected. The first reason is because a critical point may wander off into the singular locus of the variety. The study of this special locus was proposed by Bernd Sturmfels, first examples were developed by the authors of \cite{DLT15} and it was named \textit{ED data singular locus.} In a similar fashion, the second case is when a critical point becomes isotropic with respect to the Euclidean inner product; this locus will be called \textit{ED data isotropic locus}. In these two cases the number of critical points is smaller then the ED degree. Finally, a data point can have infinitely many critical points, but this phenomenon is apparently recorded by the ED discriminant, so we do not deal with it in this article. A classical example would be that there are infinitely many critical rank $2$ approximations of a matrix with two identical singular values.

In this article we aim to describe the data singular and the data isotropic loci of affine cones.

\textbf{Acknowledgments.}
The author is grateful to Jan Draisma for his support and help and to Bernd Sturmfels for the theoretical insight. The author was supported by the NWO Free Competition grant \textit{Tensors of bounded rank}.

\section{The special loci of data points}

In order to find the critical points algebraically, we consider $X$ to be a variety in $\mathbb{C}^n$ and we examine
all complex critical points of the complexified distance, induced by the standard symmetric bilinear form,
\[d_u(x)=(u-x|u-x)=\sum_{i=1}^n (u_i-x_i)^2,\]
with $x\in X^{\mathrm{reg}}$, where $X^{\mathrm{reg}}$ denotes the locus of regular points of $X$, so we only allow those critical points that are non-singular.
Let $X\subseteq \mathbb{C}^n$ be an irreducible algebraic variety of codimension $c$ with defining radical ideal $I$. If $x\in X^{\mathrm{reg}}$ is a critical point of $d_u$, then the following holds: $u-x\perp T_{x}X$. This later condition can be formulated as $x\in X^{\mathrm{reg}}$ is a critical point of $d_u$ if and only if all the $(c+1)\times (c+1)$ minors of the matrix
\[
\left(
\begin{array}{c}
u-x \\
\mathrm{Jac}_{x}(I) \\
\end{array}
\right)
\]
vanish, where $\mathrm{Jac}_{x}(I)$ is the Jacobian of $I$ at the point $x$.

We define the \textit{ED-correspondence} to be the closure of the set of all pairs $(u,x)$, such that $x\in X^{\mathrm{reg}}$ is critical to $d_u$, and we denote it by $\mathcal{E}_X\subseteq \mathbb{C}^n_u\times \mathbb{C}^n_x$. In other words $\mathcal{E}_X$ is the closure of the variety defined by
\[
\left\{(u,x)\big|u\in\mathbb{C}^n,\ x\in X^{\mathrm{reg}},\ \mathrm{rank}\left(
\begin{array}{c}
u-x \\
\mathrm{Jac}_{x}(I) \\
\end{array}
\right)\leq c\right\}
\]
We have two natural projection maps $\pi_1:\mathcal{E}_X \to \mathbb{C}^n_u$ sending $(u,x)$ to $u$ and $\pi_2:\mathcal{E}_X \to \mathbb{C}^n_x$ sending $(u,x)$ to $x$. Let $\mathrm{Sing} X$ denote the singular locus of $X$, that is the set of all points of $x\in X$ such that all the $c\times c$ minors of $\mathrm{Jac}_x(I)$ vanish.

So for a given data point $u$, the cardinality of fiber of $\pi_1$ over $u$ , $\pi_1^{-1}(u)$, measures the number of critical points.

We want to discuss the locus of exceptional data points $u$ at which the number of complex critical points is different from the ED degree. As mentioned in the introduction, there are three ways in which we can have different number of critical points then expected. The first one is because a critical point may wander off into $\mathrm{Sing}X$ due to the closure appearing in the definition of $\mathcal{E}_{X}$. This locus is called the \textit{ED data singular locus.}
\newpage
\subsection{Data singular locus}
We use the precise definition of the ED data singular locus from \cite{DLT15}, that is
\[
\pi_1(\mathcal{E}_X\cap \pi_2^{-1}(\mathrm{Sing}X)).
\]
We denote the ED data singular locus of an algebraic variety $X$ by $\mathrm{DS}(X)$ (abbreviating "data singular" locus) and we aim to describe the data singular locus of affine cones. Our main result in this section is the following theorem.

\begin{thm}\label{Main1}
Let $X\subseteq \mathbb{C}^n$ be an irreducible affine cone that is not a linear space. Then the following two inclusions hold
\[
X^{*}\subseteq_{(1)} \mathrm{DS}(X)\subseteq_{(2)} X^{*}+\mathrm{Sing}X,
\] where $X^{*}$ denotes the dual variety to $X$.
\end{thm}
We view $X^{*}$ as subset of $ \mathbb{C}^n$ via the standard symmetric bilinear form $(\cdot|\cdot)$ on $\mathbb{C}^n$.
\begin{proof}
First we prove inclusion $(1)$ for a dense subset of $X^*$. For this take $u\in X^*$, such that there exists a regular point $x_r\in X^{\mathrm{reg}}$, such that $u\perp T_{x_r}X$, that is all the $(c+1)\times (c+1)$ minors of
$\left(
   \begin{array}{c}
     u \\
     \mathrm{Jac}_{x_r}(I) \\
   \end{array}
 \right)
$ vanish, where $c$ is the codimension of $X$ and $\mathrm{Jac}_{x_r}(I)$ is the Jacobian of the (radical) ideal $I$ of $X$ at the point $x_r$. We denote an arbitrary $(c+1)\times(c+1)$ minor of this matrix by $\left(
   \begin{array}{c}
     u \\
     \mathrm{Jac}_{x_r}(I) \\
   \end{array}
 \right)_{(c+1)}
$.

We claim that $(u+\lambda x_r, \lambda x_r)\in \mathcal{E}_X$ for all real $\lambda\geq 0$. We have that if $f\in I$, homogeneous of degree $d$, then $\nabla f(\lambda x)=\lambda^d \nabla f(x)$. So if $x_r$ is a regular point then $\lambda x_r$ is also regular, for any $\lambda>0$.  Moreover we get that for any $(c+1)\times (c+1)$ minor  \[\left(
                                                                                                                              \begin{array}{c}
                                                                                                                                (u+\lambda x_r)-\lambda x_r \\
                                                                                                                                \mathrm{Jac}_{\lambda x_r}(I) \\
                                                                                                                              \end{array}
                                                                                                                            \right)_{(c+1)}=\left(
                                                                                                                              \begin{array}{c}
                                                                                                                                u\\
                                                                                                                                \mathrm{Jac}_{\lambda x_r}(I) \\
                                                                                                                              \end{array}
                                                                                                                            \right)_{(c+1)}=\lambda^N \left(
                                                                                                                              \begin{array}{c}
                                                                                                                                u \\
                                                                                                                                \mathrm{Jac}_{x_r}(I) \\
                                                                                                                              \end{array}
                                                                                                                            \right)_{(c+1)}=0,\] where $N$ is the sum of degrees of the defining polynomials of $I$.

So $(u+\lambda x_r, \lambda x_r)\in \mathcal{E}_X$ for all real $\lambda>0$.
But then taking the limit when $\lambda$ goes to zero, we get that
\[
(u,0)\in \mathcal{E}_X\cap \pi_2^{-1}(\mathrm{Sing}X),
\] since $\mathcal{E}_X\cap \pi_2^{-1}(\mathrm{Sing}X)$ is Zariski closed (hence closed wrt. Euclidean topology as well) and since $0\in \mathrm{Sing}X$. Indeed, for every $x\in X$ the line $\{\lambda \cdot x\}$ is in the tangent space to $0$, so $T_0 X$ is equal to the the linear span of $X$, which has a grater dimension that $X$ if and only if $X$ is not a linear space, hence $0\in \mathrm{Sing}X$. So then $u=\pi_1((u,0))\in \mathrm{DS}(X)$.

For the proof of $(2)$ take an element $(u,x_0)\in \mathcal{E}_X\cap \pi_2^{-1}(\mathrm{Sing}X)$, then this point can be approximated by a sequence in the part of $\mathcal{E}_X$ over $X^{\mathrm{reg}}$. That is there exists a sequence $\delta_i\rightarrow 0$ in $\mathbb{C}^n$ and $x_i\rightarrow x_0$ with all the $x_i\in X^{\mathrm{reg}}$, such that
\[
(u+\delta_i, x_i)\in \mathcal{E}_X.
\]
By the ED Duality Theorem for affine cones (see. \cite[Theorem 5.2]{DHOST13}) we get that $(u+\delta_i)-x_i\in X^{*}$, for all $i$. Now taking the limit, when $i$ goes to infinity, we get that $u-x_0\in X^{*}$, since $X^{*}$ is closed (hence closed wrt. Euclidean topology as well). Finally this means that $u\in x_0+X^{*}\subseteq \mathrm{Sing}X+X^{*}$.
\end{proof}
Note that the condition in the theorem that $X$ is not a linear space is necessary in order to prove a similar statement. Otherwise if $X$ is a linear subspace of $\mathbb{C}^n$, then it has a non-empty dual (its orthogonal complement with respect to the inner product), but its singular locus is empty, hence its data singular locus is empty as well.

\subsection{Data isotropic locus}
A second possibility for a data point $u$ to have smaller number of critical points than expected is by letting one of the critical points to become isotropic. The authors of \cite{DHOST13} define the ED degree of a projective variety in $\mathbb{P}^{n-1}$ to be the ED degree of the corresponding affine cone in $\mathbb{C}^n$, moreover given a data point $u$ the critical points to these two objects are in a one-to-one correspondence, given that non of the critical points lies in the isotropic quadric (see\cite[Lemma 2.8]{DHOST13}). In particular, the role of $Q$ exhibits that the computation of ED degree is a metric problem. This is the reason that even though in the definition of the affine $\mathcal{E}_X$ we keep the isotropic critical points, but when we pass to projective varieties we will exclude the isotropic points. This way the data isotropic locus represents the locus of data points which have different number of critical points if $X$ is considered as an affine cone or if is considered as a projective variety.

More precisely we define the \textit{ED data isotropic locus} to be
\[
\pi_1(\mathcal{E}_X\cap \pi_2^{-1}(Q\cap X),
\] where $Q=(\sum_{i=1}^n x_i^2)$ denotes the isotropic quadric with respect to the standard symmetric bilinear form.

We denote the ED data isotropic locus of an algebraic variety $X$ by $\mathrm{DI}(X)$ (abbreviating "data isotropic" locus).
We have the following theorem for the ED data isotropic locus of affine cones.
\begin{thm}\label{Main2}
Let $X\subseteq \mathbb{C}^n$ be an irreducible affine cone. Then the following two inclusions hold
\[
X^{*}\subseteq_{(1)} \mathrm{DI}(X)\subseteq_{(2)} X^{*}+(Q\cap X),
\] where $X^{*}$ denotes the dual variety to $X$.
\end{thm}
Again we view $X^{*}$ as subset of $ \mathbb{C}^n$ via the standard symmetric bilinear form $(\cdot|\cdot)$ on $\mathbb{C}^n$.
\begin{proof}
The proof follows the lines of the proof of \ref{Main1}, keeping in mind that $0\in X$ is always an isotropic point.
\end{proof}
In the following two sections we will give examples to show that both inclusions appearing in Theorem~\ref{Main1} and Theorem~\ref{Main2} can be strict and/or equalities.

\section{Examples of the ED data singular locus}
In this section we present several useful examples concerning the ED data singular locus of an affine cone. Before we get to the examples we present how can one computationally determine the objects we are working with. We illustrate the main algorithms with code in {\tt Macaulay2} \cite{M2}.
For an affine cone $X\subseteq \mathbb{C}^n$, of codimension $c$ with defining radical ideal $I$, one can determine its dual $X^*$ using the following code by \cite[Algorithm 5.1]{RS13}.
\newpage
\begin{example}[\textbf{Computing the dual variety}]
We present the algorithm for the real affine cone $X\subseteq \mathbb{C}^{3}$ defined by the homogeneous equation $f=x_1^3 + x_2^2 x_3$.
\begin{verbatim}
n=3;
kk=QQ[x_1..x_n,y_1..y_n];
f=x_1^3+x_2^2*x_3;
I=ideal(f);
c=codim I;
Y=matrix{{y_1..y_n}};
Jac= jacobian gens I;
S=submatrix(Jac,{0..n-1},{0..numgens(I)-1});
Jbar=S|transpose(Y);
EX = I + minors(c+1,Jbar);
SingX=I+minors(c,Jac);
EXreg=saturate(EX,SingX);
IDual=eliminate(toList(x_1..x_n),EXreg)
\end{verbatim}
Which gives at the end that $X^*$ is the zero locus of the polynomial $f^*=4x_1^3 - 27x_2^2x_3$
\end{example}
Following the definition of the data singular locus, the next example contains an algorithm for calculating the ideal of it.

\begin{example}[\textbf{Computing the data singular locus}]
We present the algorithm for the real affine cone $X\subseteq \mathbb{C}^{3}$ defined by the homogeneous equation $f=x_1^3 + x_2^2 x_3$.
\begin{verbatim}
n=3;
kk=QQ[x_1..x_n,y_1..y_n];
f=x_1^3+x_2^2*x_3;
I=ideal(f);
c=codim I;
Y=matrix{{x_1..x_n}}-matrix{{y_1..y_n}};
Jac= jacobian gens I;
S=submatrix(Jac,{0..n-1},{0..numgens(I)-1});
Jbar=S|transpose(Y);
EX = I + minors(c+1,Jbar);
SingX=I+minors(c,Jac);
EXreg=saturate(EX,SingX);
DSX=radical eliminate(toList(x_1..x_n),EXreg+SingX)
\end{verbatim}
Which gives as output that $\mathrm{DS}(X)$ is the zero locus of the polynomial $x_1(4x_1^3 - 27x_2^2x_3)$.
\end{example}
Now we arrived at the point to present a sequence of interesting varieties and the corresponding duals and data singular loci. The first example is the one we used for presenting the algorithms previously. In this example both inclusion $(1)$ and inclusion $(2)$ are strict, as it will be seen.
\begin{example}[\textbf{Cuspidal Cubic Cone}] Let $X\subseteq \mathbb{C}^{3}$ be the real variety defined by the homogeneous equation $f=x_1^3 + x_2^2 x_3$. Since it is an affine cone it has a dual $X^{*}$, which is defined by the dual equation $f^{*}=4x_1^3 - 27x_2^2x_3.$ For the data singular locus we get that $\mathrm{DS}(X)$ is the zero locus of the polynomial $x_1(4x_1^3 - 27x_2^2x_3)$. So we can see that $X^{*}$ is even a component of $\mathrm{DS}(X)$. Moreover $X^{*}+\mathrm{Sing} X$ is something much larger and not equal to $\mathrm{DS}(X)$. For example the point $(3,2,1)+(0,0,1)\in X^{*}+\mathrm{Sing}X$, but is not on $\mathrm{DS}(X)$. Figure $1$ shows $X$ in blue and $X^{*}$ in green and $\mathrm{DS}(X)$ is the union of the green colored $X^{*}$ and the additional surface in red.
\begin{figure}[h]\label{Fig:One}
\begin{center}
\vskip -0.3cm
\includegraphics[scale=0.7]{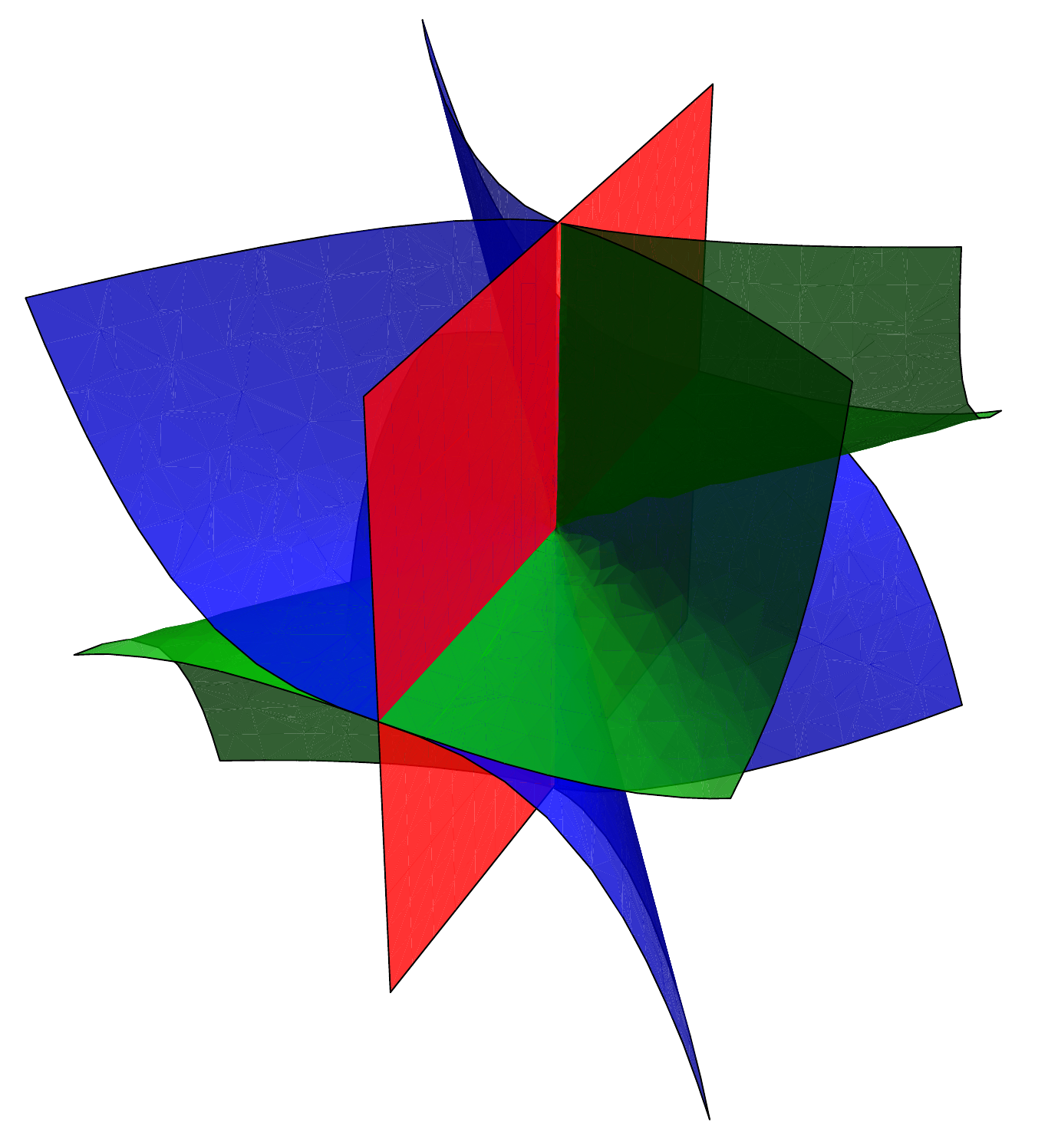}
\vskip -0.4cm
\caption{$V(x_1^3 + x_2^2x_3)$ together with its dual and its data singular locus}
\end{center}
\end{figure}
\end{example}
The next example shows that both inclusions $(1)$ and $(2)$ can be in fact equalities. More generally we have the following corollary to Theorem~\ref{Main1}.
\begin{cor}\label{Cor:Smooth}
Let $X\subseteq \mathbb{C}^n$ be an affine cone, with $\mathrm{Sing}X=\{0\}$, then $\mathrm{DS}(X)=X^{*}.$ Moreover if $X$ is a general hypersurface of degree $d$, then \[\mathrm{deg}(\mathrm{DS}(X))=d(d-1)^{n-1}.\]
\end{cor}
\begin{proof}
The first part follows directly from the claim of Theorem~\ref{Main1}. The moreover part is classical and we refer to \cite[Exercise 5.14]{RS13}.
\end{proof}
\begin{example}[\textbf{Cone over ellipse}]
Let $X\subseteq \mathbb{C}^{3}$ the cone over an ellipse, defined by the homogeneous equation $f=x_1^2 + 4x_2^2 - 9x_3^2$. The singular locus $\mathrm{Sing}X$ only contains $0$, so as a consequence of Theorem~\ref{Main1} we have that $\mathrm{DS}(X)$ equals the dual variety $X^{*}$, defined by the dual equation $f^*=x_1^2 + x_2^2/4 - x_3^2/9$. Figure $2$ shows $X$ in blue and $X^*$ in green.
\begin{figure}[h]
\begin{center}
\vskip -0.3cm
\includegraphics[scale=0.7]{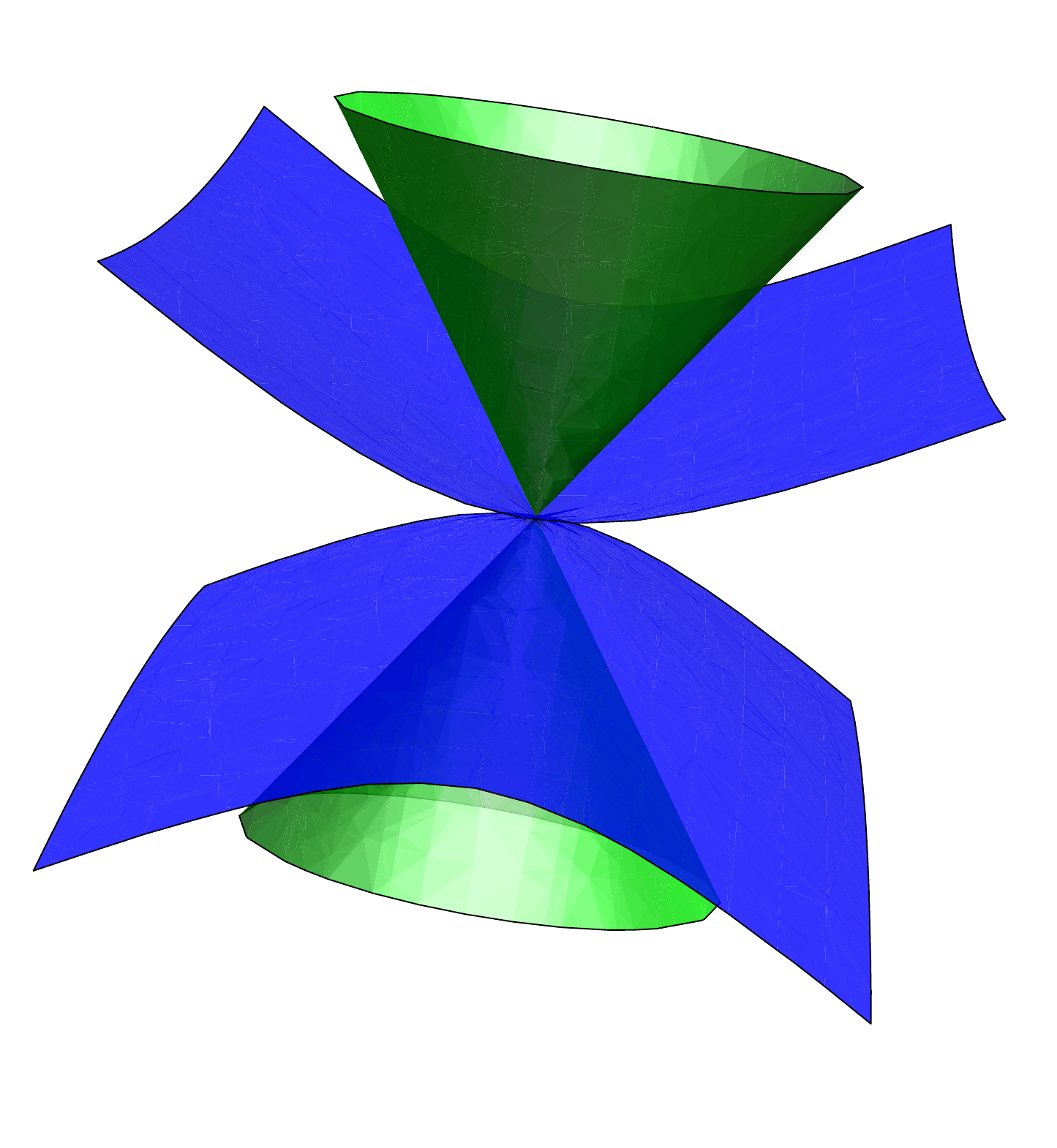}
\vskip -0.4cm
\caption{$V(x_1^2 + 4x_2^2 - 9x_3^2)$ together with its dual}
\end{center}
\end{figure}
\end{example}
The next example concernes the well known and much used \textit{determinantal varieties}. We will see that for this variety inclusion $(1)$ is strict and inclusion $(2)$ is an equality.
\begin{example}[\textbf{Determinantal varieties}]\label{Ex:Determinant}
Denote by $M_{n\times m}^{\leq r}$ the variety of $n\times m$ matrices (suppose $n\leq m$) of rank at most $r$. It is classical that the singular locus is the variety $M_{n\times m}^{\leq r-1}$. By \cite[Chapter 1, Prop. 4.11]{GKZ} we have that the dual variety is exactly $M_{n\times m}^{\leq n-r}$. So applying Theorem~\ref{Main1} we get that
\[
M_{n\times m}^{\leq n-r}\subseteq \mathrm{DS}(M_{n\times m}^{\leq r})\subseteq M_{n\times m}^{\leq n-r}+M_{n\times m}^{\leq r-1}=M_{n\times m}^{\leq n-1}.
\] So for rank-one matrices ($r=1$) we get that $\mathrm{DS}(M_{n\times m}^{\leq 1})=M_{n\times m}^{\leq n-1}$, which is not a surprise based on Corollary~\ref{Cor:Smooth}, since $M_{n\times m}^{\leq 1}$ is smooth, except $0$.  But something more is true for general $r$. We claim that the upper bound for the inclusions is always attained. For this we have the following proposition.
\end{example}

\begin{prop}
The ED data singular locus of the determinantal variety $M_{n\times m}^{\leq r}$ is equal to $M_{n\times m}^{\leq n-1}$, for all $1\leq r\leq n-1$.
\end{prop}
\begin{proof}
A $n\times m$ matrix $U$ lies in $\mathrm{DS}(M_{n\times m}^{\leq r})$ if and only if it has a singular critical point. By \cite[Example 2.3]{DHOST13} all the critical points of $U$ look like \[T_1\cdot \mathrm{Diag}(0,0,...,\sigma_{i_1},0,...,0,\sigma_{i_r},0,...,0)\cdot T_2,\] where the singular value decomposition of $U$ is equal to $U=T_1\cdot\mathrm{Diag}(\sigma_1,...,\sigma_n)\cdot T_2$, with $\sigma_1>...>\sigma_n$ singular values and $T_1,T_2$ orthogonal matrices of size $n\times n$ and $m\times m$. Such a critical point is singular if and only if it has rank at most $r-1$, which can only happen if one of the singular values $\sigma_{i_1},...,\sigma_{i_r}$ is zero. So there exists a  singular critical point to $U$ if and only if there is a zero singular value of $U$, which can only happen if $U$ has a rank defect, hence all the $(n-1)\times (n-1)$ minors are zero, that is $U\in M_{n\times m}^{\leq n-1}$.
\end{proof}
The next example shows that $X^*$ is a subvariety of $\mathrm{DS}(X)$ but not necessarily a component of it.
\begin{example}[\textbf{Hurwitz determinant}]\label{Ex:Hurwitz}
In control theory, to check whether a given polynomial is stable one builds up the so called Hurwitz matrix $H_n$ and checks if every leading principal minor of $H_n$ is positive. Take $n=4$, then the $4$-th Hurwitz matrix looks like
\[
H_4=\left(
      \begin{array}{cccc}
        x_2 & x_4 & 0 & 0 \\
        x_1 & x_3 & x_5 & 0 \\
        0 & x_2 & x_4 & 0 \\
        0 & x_1 &x_3  & x_5 \\
      \end{array}
    \right).
\]
The ratio $\Gamma_4 = \det(H_4)/x_5$ is a
homogeneous polynomial and it is called the \textit{Hurwitz determinant} for $n=4$ by \cite[Example 3.5]{DHOST13}.

Let $X\subseteq \mathbb{C}^5$ be the affine cone defined by $\Gamma_4$. Then its dual variety has one irreducible component given by
\[
X^*=V(-x_3x_4+x_2x_5,-x_3^2+x_1x_5,-x_2x_3+x_1x_4).
\]
While its data singular locus $\mathrm{DS}(X)$ has two irreducible components and it is defined by
\[
V((x_1x_2^2+x_2x_3x_4+x_4^2x_5)(x_2^4x_3-x_1x_2^3x_4-2x_1x_2x_4^3-x_3x_4^4+2x_2^3x_4x_5+x_2x_4^3x_5)).
\]
It is clear that $X^*$ is not component of $\mathrm{DS}(X)$. Moreover $\mathrm{DS}(X)$ is not equal to $X^*+\mathrm{Sing}X$, since $\mathrm{Sing}X=V(x_2,x_4)$ and the point \[(2,1,1,0,1)=(1,1,0,0,0)+(1,0,1,0,1)\] lies on $X^*+\mathrm{Sing}X$ but it is not on $\mathrm{DS}(X)$.
\end{example}
We have thus seen examples of varieties with: both inclusions in Theorem~\ref{Main1} being strict, both inclusions in Theorem~\ref{Main1} being equalities and the second inclusion being an equality, while the first one is strict. It is natural to ask if there are examples where the first inclusion is an equality, while the second one is strict. The author could not find such an example, so the following question arises.
\begin{problem}
Find an affine cone $X$, such that $X^*=\mathrm{DS}\subset X^*+\mathrm{Sing}(X)$ or prove that there is no such $X$.
\end{problem}

\section{Examples of the ED data isotropic locus}
In this section we present several application oriented examples concerning the ED data isotropic locus of an affine cone. We begin with presenting how can one computationally determine the data isotropic locus of a variety.

\begin{example}[\textbf{Computing the data isotropic locus}]
We present the algorithm for the real affine cone $X\subseteq \mathbb{C}^{6}$ defined by the homogeneous equation $f=x_1x_6-x_2x_5+x_3x_4$, representing the Grassmanian of planes in $4$-space.

\begin{verbatim}
n=6;
kk=QQ[x_1..x_n,y_1..y_n];
f=x_1*x_6-x_2*x_5+x_3*x_4;
I=ideal(f);
c=codim I;
Y=matrix{{x_1..x_n}}-matrix{{y_1..y_n}};
Jac= jacobian gens I;
S=submatrix(Jac,{0..n-1},{0..numgens(I)-1});
Jbar=S|transpose(Y);
EX = I + minors(c+1,Jbar);
SingX=I+minors(c,Jac);
q=sum for i from 1 to n list x_i^2;
Q=ideal(q);
EXreg=saturate(EX,SingX);
DIX=radical eliminate(toList(x_1..x_n),EXreg+Q)
\end{verbatim}
Which gives that $\mathrm{DI}(X)$ is the zero locus of the polynomial $x_1x_6-x_2x_5+x_3x_4$, so we get that the data isotropic locus is equal to the dual variety which equals the variety.
\end{example}

The next example shows that the data isotropic locus can be equal to the dual and strictly contained in $X^*+(X\cap Q)$.
\begin{example}[\bf{Cayley-Menger variety}]\label{Ex:Cayley-Menger}
Let $X$ denote the variety in $\mathbb{C}^3$ with parametric representation
\[\left\{
  \begin{array}{ll}
    x_1 = (z_1-z_2)^2,  \\
    x_2=(z_1-z_3)^2, \\
    x_3=(z_2-z_3)^2.
  \end{array}
\right.
\]
Based on \cite{AH14} and on \cite[Example 3.7]{DHOST13}, the points in $X$ record the squared distances among $3$ interacting agents with coordinates
$z_1, z_2$ and $z_3$ on the line $\mathbb{R}$. The prime ideal of $X$ is given by the determinant of the \textit{Cayley-Menger matrix}
\[
\left(
  \begin{array}{cc}
    2x_2 & x_2+x_3-x_1 \\
    x_2+x_3-x_1 & 2x_3 \\
  \end{array}
\right)
\]
So $X$ is defined by the irreducible polynomial \[f=x_1^2-2x_1x_2+x_2^2-2x_1x_3-2x_2x_3+x_3^2.\]
After running the computations one can see that the data isotropic locus equals the dual variety, which is defined by $f^*=x_1x_2+x_1x_3+x_2x_3$, see Figure $3$. And it does not equal $X^*+(Q\cap X)$ for example because the point $(1,0,0)+(0,1,i)\in X^*+(Q\cap X)$, but it does not lie on $\mathrm{DI}(X)$.
\end{example}
\begin{figure}[h]
\begin{center}
\vskip -0.3cm
\includegraphics[scale=0.7]{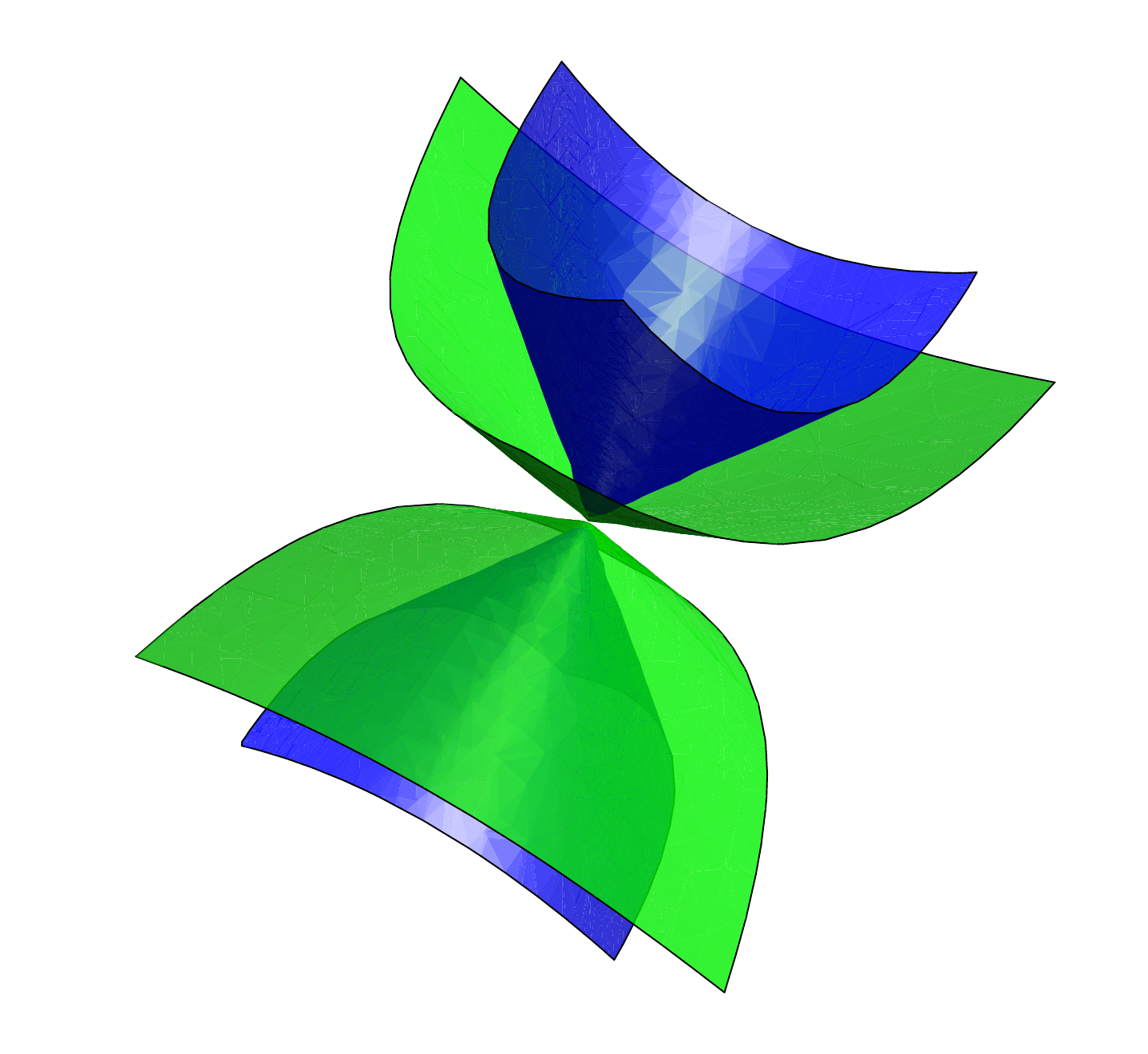}
\vskip -0.4cm
\caption{Cayley-Menger variety (in blue) together with its dual (in green).}
\end{center}
\end{figure}

The next example shows that both inclusions from Theorem~\ref{Main2} can be strict.
\begin{example}[\bf{Cayley's Cubic}]\label{Ex:Cayley cubic}
Let $X$ be defined by $f=x_1^3-x_1x_2^2-x_1x_3^2+2x_2x_3x_4-x_1x_4^2$, the $3\times 3$ symmetric determinant in $\mathbb{C}^4$. This hypersurface is sometimes called the \textit{Cayley's cubic surface} and receives much attention in the study of elliptopes and exponential varieties in algebraic statistics, see for instance \cite[Example 5.44]{RS13}, \cite[Example 1.1]{MSUZ15} and \cite{LP96}.
Its dual variety is the \textit{quartic Steiner surface} defined by $f^*=x_2^2x_3^2-2x_1x_2x_3x_4+x_2^2x_4^2+x_3^2x_4^2$. After running the computations one finds that the data isotropic locus is the union
\[
\mathrm{DI}(X)=V(x_1^{18}+4x_1^{16}x_2^2+6x_1^{14}x_2^4-...+729x_3^4x_4^{14})\cup X^*.
\]
So it is clearly not equal to the dual variety. And it is not equal to the $X^*+(Q\cap X)$ either, because for example the point
\[
(1,1,0,0)+(0,0,1,i)\in X^*+(Q\cap X),
\] but it is not in $\mathrm{DI}(X)$.
\end{example}
Our next example shows that the second inclusion in Theorem~\ref{Main2} can be equality and moreover it can give the whole space.
\begin{example}[\bf{Special essential variety}]\label{Ex:Essential}
Essential matrices play an important role in \textit{multiview geometry}, see for instance \cite{HZ03}. The connections between the ED degree theory and multiview geometry were investigated in \cite[Example 3.3]{DHOST13}. The set of essential matrices is called the \textit{essential variety} and it is defined as follows
\[
\mathcal{E}=\{X\in M_{3\times 3}| \det X=0,\ 2XX^{T}X-\text{trace}(XX^{T})X=0\}.
\]
It is a codimension $3$ variety of degree $10$. The ED degree of $\mathcal{E}$ is $6$, as was proved in \cite[Example 5.8]{DLT15}. We are interested in the data isotropic locus of this variety, but because of computational reasons we will take a linear section of it and we will only consider the symmetric, constant diagonal essential matrices , which we will call the special essential variety and will denote by $\mathcal{SE}$. More precisely we define $\mathcal{SE}$ to be
\[
\Bigg\{ X=\left(
                   \begin{array}{ccc}
                     x_1 & x_2 & x_3 \\
                     x_2 & x_1 & x_4 \\
                     x_3 & x_4 & x_1 \\
                   \end{array}
                 \right)
\Bigg | \det X=0,\ 2XX^{T}X-\text{trace}(XX^{T})X=0 \Bigg\}.
\]
Since this variety will not be irreducible we will carry out our computations componentwise. When running the computations one will find that the data isotropic locus is the hole space. Indeed one can observe that $\mathcal{SE}$ is inside the isotropic quadric $Q$, so every critical point is isotropic. We have that
\[ \mathrm{DI}(X)=X^*+(X\cap Q)=X^*+X=\mathbb{C}^4.\]
Moreover $\mathrm{DI}(X)$ is not equal to the dual variety, since $X^*$ is a proper variety defined by $f^*=(x_2^2+x_4^2)(x_2^2+x_3^2)(x_3^2+x_4^2)$. Moreover it is clear that the dual is not a component of $\mathrm{DI}(X)$.
\end{example}
In the last example the reader can see that both inclusions from Theorem~\ref{Main2} can be equalities.
\begin{example}[\bf{Line through the origin}]
In what follows let $X$ be a line through the origin in $\mathbb{C}^3$. So we have $X=V(x_1+2x_2+3x_3, 4x_1+5x_2+6x_3)$.
Then we get that $X$ intersects the quadric $Q$ only in the point $0$, so by Theorem~\ref{Main2} we get immediately, that $X^*=\mathrm{DI}(X)=X^*+\{0\}$, and the dual is the orthogonal complement of $X$, so it is defined by $x_1-2x_2+x_3$.
\end{example}

\end{document}